\documentclass[a4paper, 12pt]{amsart}
\usepackage[cp1251]{inputenc}
\usepackage[english,russian]{babel}
\usepackage{amsmath,amssymb,a4wide,latexsym,amscd,amsthm}
\usepackage{graphics}

\theoremstyle{plain}
\newtheorem{theorem}{Теорема}[section]
\newtheorem{lemma}{Лемма}[section]
\newtheorem{proposition}{Утверждение}[section]
\newtheorem{corollary}{Следствие}[section]

\theoremstyle{definition}
\newtheorem{example}{Пример}[section]
\newtheorem{definition}{Определение}[section]
\newtheorem{remark}{Замечание}

\numberwithin{equation}{section}

\begin{document}

УДК~\textbf{517.982.22}

\title[]
{О представляющих и абсолютно представляющих системах подпространств в банаховых пространствах}

\author{И.С.Фещенко}

\maketitle

\begin{abstract}
В работе изучаются свойства
представляющих систем подпространств и абсолютно представляющих систем подпространств
в банаховых пространствах.

\textbf{Ключевые слова:} банахово пространство, представляющая система подпространств, абсолютно представляющая система подпространств.
\end{abstract}

\section{Введение}

Пусть
$X$ "--- линейное нормированное пространство над полем
$\mathbb{K}$ действительных или комплексных чисел,
$X_{k},k\geqslant 1,$ "--- система подпространств $X$ (т.е. замкнутых линейных множеств), которую мы будем обозначать
$S=(X;X_{k},k\geqslant 1).$

\begin{definition}(\cite{KorRSofSubspaces})
Система $S$
называется представляющей системой подпространств (ПСП) в
$X$, если произвольный
$x\in X$ можно представить в виде
$x=\sum_{k=1}^{\infty}x_{k}$, где
$x_{k}\in X_{k},k\geqslant 1$.
\end{definition}

Отметим, что в случае банахова пространства
$X$ в книге \cite{Singer} (определение 15.21) ПСП в $X$ называется псевдоразложением $X$.

\begin{definition}(\cite{KorRSofSubspaces})
Система $S$ называется абсолютно представляющей системой подпространств в $X$, если произвольный
$x\in X$ можно представить в виде
$x=\sum_{k=1}^{\infty}x_{k}$, где
$x_{k}\in X_{k},k\geqslant 1$ и
$\sum_{k=1}^{\infty}\|x_{k}\|<\infty$.
\end{definition}

ПСП и АПСП являются естественным обобщением представляющих и абсолютно представляющих систем
(все $X_{k}$ одномерны)
(см., например, \cite{KorRS}).
Определения ПСП и АПСП можно давать и в более широких классах пространств, чем линейные нормированные
(см.\cite{KorRSofSubspaces}); ПСП и АПСП в различных классах пространств изучались в работах
Ю.Ф. Коробейника, А.В. Абанина, К.А.Михайлова и др. (см., например, \cite{KorRSofSubspaces},\cite{Abanin},\cite{Mihaylov}).
В данной работе мы изучаем ПСП и АПСП в банаховых пространствах.

В параграфе \ref{S:RSofSub} изучаются свойства ПСП в $X$, приведены достаточные условия для того, чтобы
счётная система подпространств была ПСП в $X$, доказана теорема об устойчивости ПСП, а также показана связь
между ПСП в гильбертовом пространстве и проблемой Гальперина.

В параграфе \ref{S:ARSofSub} приводятся различные критерии АПСП, изучаются свойства АПСП в $X$.

\section{Представляющие системы подпространств в банаховых пространствах}\label{S:RSofSub}

\subsection{Некоторые свойства ПСП в банаховых пространствах}\label{SS:PropRSofSub}

Пусть
$X$ "--- банахово пространство,
$S=(X;X_{k},k\geqslant 1)$ "--- система его подпространств.
Введём необходимые обозначения и определения.
Для непустого
$I\subset\mathbb{N}$ определим подпространство
$S^{(I)}$ как замыкание линейной оболочки подпространств
$X_{i},i\in I$. Обозначим $X^{*}$ сопряженное пространство к $X$; для
$\varphi\in X^{*}$ обозначим
$\varphi^{(S,I)}$ сужение
$\varphi$ на
$S^{(I)}$. Ясно, что
$\varphi^{(S,I)}\in (S^{(I)})^{*}$.
Для разбиения
$\pi=\{I_{k}\}$ ($k$ пробегает конечное или счётное число значений) множества
$\mathbb{N}$ определим
\begin{equation*}
F_{1}(S,\pi,\varphi)=\sum_{k}\|\varphi^{(S,I_{k})}\|,\varphi\in X^{*}.
\end{equation*}

Разбиение
$\pi$ множества $\mathbb{N}$ назовём последовательным, если оно имеет один из следующих двух видов:

$(1)$ Множества
$I_{k}=\{n(k-1)+1,n(k-1)+2,\ldots,n(k)\},k\geqslant 1$ для некоторой возрастающей последовательности натуральных чисел
$n(1)<n(2)<\ldots$, $n(0)=0$.

$(2)$
Множества
$I_{k}=\{n(k-1)+1,\ldots,n(k)\},1\leqslant k\leqslant p$,
$I_{p+1}=\{n(p)+1,n(p)+2,\ldots\}$ для некоторой возрастающей последовательности натуральных чисел
$n(1)<\ldots<n(p)$, $n(0)=0$.

\begin{theorem}\label{T:RSinequality}
Пусть
$S$ "--- ПСП в
$X$. Тогда существует
$\varepsilon>0$ такое, что для
произвольного последовательного разбиения $\pi$ множества $\mathbb{N}$
выполнено
\begin{equation}\label{E:RSinequality}
F_{1}(S,\pi,\varphi)\geqslant\varepsilon\|\varphi\|,\,\varphi\in X^{*}.
\end{equation}
\end{theorem}
\begin{proof}
Определим пространство
\begin{equation*}
D_{c}=\{\xi=(x_{1},x_{2},\ldots)\mid x_{k}\in X_{k},k\geqslant 1,\,\sum_{k=1}^{\infty}x_{k}\,\mbox{сходится}\}
\end{equation*}
с нормой
$\|\xi\|=\sup_{k\geqslant 1}\|x_{1}+\ldots+x_{k}\|$.
Легко проверить, что $D_{c}$ банахово.
Определим линейный оператор
$A:D_{c}\to X$ равенством
$A(x_{1},x_{2},\ldots)=\sum_{k=1}^{\infty}x_{k}.$
Тогда $A$ ограничен (более того, $\|A\|\leqslant 1$).
Поскольку
$S$ является ПСП в
$X$, то
$\mathrm{Im}(A)=X$.
Из теоремы про открытое отображение следует существование
числа
$M>0$ такого, что для произвольного
$x\in X$ существует
$\xi=(x_{1},x_{2},\ldots)\in D_{c}$, для которого
$x=A\xi$ и
$\|\xi\|\leqslant M\|x\|$. Тогда
$x=\sum_{k=1}^{\infty}x_{k}$ и для каждого
$k\geqslant 1$
$\|x_{1}+\ldots+x_{k}\|\leqslant M\|x\|$.
Для натуральных
$l\leqslant k$ определим
$x_{l,k}=x_{l}+\ldots+x_{k}$, тогда
$\|x_{l,k}\|\leqslant 2M\|x\|$.
Для последовательного разбиения $\pi$ первого вида и
$\varphi\in X^{*}$ имеем
\begin{align*}
&|\varphi(x)|=|\varphi(\sum_{k=1}^{\infty}x_{n(k-1)+1,n(k)})|\leqslant\sum_{k=1}^{\infty}|\varphi(x_{n(k-1)+1,n(k)})|\leqslant\\
&\leqslant\sum_{k=1}^{\infty}\|\varphi^{(S,I_{k})}\|\|x_{n(k-1)+1,n(k)}\|\leqslant 2M\|x\|F_{1}(S,\pi,\varphi),
\end{align*}
откуда, в силу произвольности
$x\in X$,
$F_{1}(S,\pi,\varphi)\geqslant(1/(2M))\|\varphi\|$. Для последовательного разбиения
$\pi$ второго вида такая же оценка доказывается аналогично.
\end{proof}

Для подмножества
$M\subset X$ обозначим
$M^{\bot}$ множество всех
$\varphi\in X^{*}$ таких, что
$\varphi(x)=0,x\in M$. Напомним (см., например, параграф 15 книги \cite{Singer}), что система подпространств
$G_{k},k\geqslant 1,$ банахова пространства $E$ называется разложением Шаудера
$E$, если для каждого
$x\in E$ существуют и единственны
$x_{k}\in G_{k},k\geqslant 1$, такие, что
$x=\sum_{k=1}^{\infty}x_{k}$.

\begin{theorem}\label{T:RSbiorthogonal}
Пусть
$S$ "--- ПСП в
$X$. Тогда система подпространств
$X_{k}'=\bigcap_{j\neq k}X_{j}^{\bot},k\geqslant 1$
является разложением Шаудера в замыкании своей линейной оболочки.
\end{theorem}

\begin{proof}
Достаточно доказать, что существует
$\varepsilon>0$, такое, что для произвольных натуральных
$n,m$ и произвольных
$\varphi\in\sum_{k=1}^{n}X_{k}',\,\psi\in\sum_{k=n+1}^{n+m}X_{k}'$
выполнено
$\|\varphi+\psi\|\geqslant\varepsilon\|\varphi\|$ (см. теорему 15.5 в \cite{Singer}).

Из теоремы \ref{T:RSinequality} следует, что существует
$\varepsilon>0$ такое, что для произвольного последовательного разбиения
$\pi$ выполнено неравенство
\eqref{E:RSinequality}.
Пусть $\varphi\in\sum_{k=1}^{n}X_{k}',\,\psi\in\sum_{k=n+1}^{n+m}X_{k}'$. Определим разбиение
$\pi$ так:
$I_{1}=\{1,2,\ldots,n\},I_{2}=\{n+1,n+2,\ldots\}$.
Тогда
\begin{equation*}
\|\varphi+\psi\|\geqslant\|(\varphi+\psi)^{(S,I_{1})}\|=\|\varphi^{(S,I_{1})}\|=\|\varphi^{(S,I_{1})}\|+\|\varphi^{(S,I_{2})}\|\geqslant\varepsilon
\|\varphi\|,
\end{equation*}
откуда следует нужное утверждение.
\end{proof}

Будем говорить, что система подпространств
$S$ является перестановочной ПСП (ППСП) в $X$, если для произвольной биекции $\sigma:\mathbb{N}\to \mathbb{N}$
система подпространств
$S_{\sigma}=(X;X_{\sigma(k)},k\geqslant 1)$ является ПСП в
$X$. Напомним (см., например, с.534 в \cite{Singer}), что разложение Шаудера
$G_{k},k\geqslant 1$ банахова пространства $E$
называется безусловным, если каждый сходящийся ряд вида
$\sum_{k=1}^{\infty}x_{k}$,
$x_{k}\in G_{k},k\geqslant 1,$ сходится безусловно.

\begin{theorem}
Пусть
$S$ "--- ППСП в
$X$. Тогда система подпространств
$X_{k}'=\bigcap_{j\neq k}X_{j}^{\bot},k\geqslant 1$
является безусловным разложением Шаудера в замыкании своей линейной оболочки.
\end{theorem}
\begin{proof}
Воспользуемся следующим утверждением (см. теорему 15.18 в \cite{Singer}):
если
$G_{k},k\geqslant 1,$ "--- система подпространств банахова пространства
$E$, причём замыкание линейной оболочки $G_{k},k\geqslant 1,$ равно
$E$, то $G_{k},k\geqslant 1,$ является безусловным разложением Шаудера
$E$ тогда и только тогда, когда для произвольной биекции
$\sigma:\mathbb{N}\to\mathbb{N}$ система подпространств
$G_{\sigma(k)},k\geqslant 1,$ является разложением Шаудера
$E$. Теперь из теоремы
\ref{T:RSbiorthogonal} следует нужное утверждение.
\end{proof}

\subsection{Достаточное условие для того, чтобы система подпространств была ПСП в $X$}

Пусть
$X$ "--- линейное нормированное пространство,
$S=(X;X_{k},k\geqslant 1)$ "--- система его подпространств.
Для множества
$F\subset X$ и
элемента $x\in X$ обозначим
$d(x,F)$ расстояние от $x$ до $F$.
Определим множество
\begin{equation*}
\Delta(S)=\{x\in X\mid \liminf_{k\to\infty} d(x,X_{k})=0\}.
\end{equation*}

\begin{theorem}\label{T:SUFforRS}
Если замыкание линейной оболочки $\Delta(S)$ равно $X$, то
$S$ является ППСП в $X$.
\end{theorem}
\begin{proof}
Очевидно, для произвольной биекции
$\sigma:\mathbb{N}\to\mathbb{N}$
$\Delta(S_{\sigma})=\Delta(S)$. Поэтому достаточно доказать, что
произвольная система подпространств $S$, удовлетворяющая условию теоремы,
является ПСП в $X$.

Достаточно доказать, что произвольный
$x\in X,\|x\|< 1$ может быть представлен в виде
$x=\sum_{k=1}^{\infty}x_{k}$, где
$x_{k}\in X_{k}$. Итак, пусть
$x\in X,\|x\|<1$. Для
$k=1,2,\ldots$ проделаем следующие операции.

Пусть для некоторого
$k\geqslant 1$ у нас уже определены натуральные числа
$N(l,i,j)$ для
$l=1,\ldots,k-1$,
$i=1,\ldots,r(l)$,
$j=1,\ldots,N(l)$ и
элементы $y_{l,i,j}\in X_{N(l,i,j)}$, причём
\begin{equation*}
\|x-\sum_{l=1}^{k-1}\sum_{i=1}^{r(l)}\sum_{j=1}^{N(l)}y_{l,i,j}\|<2^{-(k-1)}
\end{equation*}
(для $k=1$ ничего не определено).
Обозначим
$z=x-\sum_{l=1}^{k-1}\sum_{i=1}^{r(l)}\sum_{j=1}^{N(l)}y_{l,i,j}$, тогда
$\|z\|<2^{-(k-1)}$ (для $k=1$ определяем $z=x$).
Существует
$N(k)\in\mathbb{N}$ и элементы
$x_{k,j}\in\Delta(S),j=1,\ldots,N(k)$, такие, что
\begin{equation}\label{E:sufk1}
\|z-\sum_{j=1}^{N(k)}x_{k,j}\|<2^{-k}.
\end{equation}
Выберем
$r(k)\in\mathbb{N}$ так, чтобы
\begin{equation}\label{E:sufk2}
\|x_{k,j}/r(k)\|<2^{-k}(N(k))^{-1},~
j=1,\ldots,N(k).
\end{equation}
Неравенство \eqref{E:sufk1} перепишем в виде
\begin{equation}\label{E:sufk3}
\|z-\Bigl(\underbrace{\dfrac{x_{k,1}+\ldots+x_{k,N(k)}}{r(k)}+\ldots+\dfrac{x_{k,1}+\ldots+x_{k,N(k)}}{r(k)}}_{r(k)}\Bigl)\|<2^{-k}.
\end{equation}
Из $\|z\|<2^{-(k-1)}$ и
\eqref{E:sufk1} следует, что
$\|\sum_{j=1}^{N(k)}x_{k,j}\|<2^{-(k-1)}+2^{-k}$. Поэтому для произвольного
$a=1,\ldots,r(k)-1$ имеем:
\begin{equation}\label{E:sufk4}
\|\underbrace{\dfrac{x_{k,1}+\ldots+x_{k,N(k)}}{r(k)}+\ldots+\dfrac{x_{k,1}+\ldots+x_{k,N(k)}}{r(k)}}_{a}\|<2^{-(k-1)}+2^{-k}
\end{equation}
Ясно, что
$x_{k,j}/r(k)\in\Delta(S)$ для
$j=1,\ldots,N(k)$.
Используя неравенства
\eqref{E:sufk2},\eqref{E:sufk3},\eqref{E:sufk4}
несложно видеть, что существуют натуральные числа
$N(k,i,j),i=1,\ldots,r(k);j=1,\ldots,N(k)$ и элементы
$y_{k,i,j}\in X_{N(k,i,j)}$, такие, что

$\textbf{(1)}~$
$N(k-1,r(k-1),N(k-1))<N(k,1,1)$ (для $k=1$ это условие отсутствует);
$N(k,i,j)<N(k,i',j')$ если
$i<i'$;
$N(k,i,j)<N(k,i,j')$ если
$j<j'$;

$\textbf{(2)}~$ $\|z-\sum_{i=1}^{r(k)}\sum_{j=1}^{N(k)}y_{k,i,j}\|<2^{-k}$, т.е.
\begin{equation*}
\|x-\sum_{l=1}^{k}\sum_{i=1}^{r(l)}\sum_{j=1}^{N(l)}y_{l,i,j}\|<2^{-k};
\end{equation*}

$\textbf{(3)}~$
для произвольного
$a=1,\ldots,r(k)-1$
\begin{equation*}
\|\sum_{i=1}^{a}\sum_{j=1}^{N(k)}y_{k,i,j}\|<2^{-(k-1)}+2^{-k};
\end{equation*}

$\textbf{(4)}~$ $\|y_{k,i,j}\|<2^{-k}(N(k))^{-1}$ для всех $i=1,2,\ldots,r(k),j=1,2,\ldots,N(k)$.

Действительно, сначала выберем
$N(k,1,1)$ и
$y(k,1,1)$ (достаточно близко к $x_{k,1}/r(k)$), затем
$N(k,1,2)$ и
$y_{k,1,2}$ (достаточно близко к $x_{k,2}/r(k)$),
$\ldots$, затем
$N(k,1,N(k))$ и
$y_{k,1,N(k)}$ (достаточно близко к $x_{k,N(k)}/r(k)$),
затем переходим к выбору <<второй группы>>:
$N(k,2,1)$ и
$y(k,2,1)$ (достаточно близко к $x_{k,1}/r(k)$) и т.д.

Выполнив такие операции, получим набор элементов
$y_{l,i,j}\in X_{N(l,i,j)}$, где
$l=1,2,\ldots$,
$i=1,\ldots,r(l)$,
$j=1,2\ldots,N(l)$. По построению
$N(l,i,j)<N(l',i',j')$ если
$l<l'$;
$N(l,i,j)<N(l,i',j')$ если
$i<i'$;
$N(l,i,j)<N(l,i,j')$ если
$j<j'$.
Покажем, что
\begin{align}\label{E:sufconv}
&x=y_{1,1,1}+y_{1,1,2}+\ldots+y_{1,1,N(1)}+y_{1,2,1}+\ldots+y_{1,2,N(1)}+\ldots+\\
&+y_{1,r(1),1}+\ldots+y_{1,r(1),N(1)}+y_{2,1,1}+\ldots+y_{2,1,N(2)}+y_{2,2,1}+\ldots
\notag
\end{align}
Для этого рассмотрим сумму первых
$s$ членов ряда в правой части \eqref{E:sufconv}.
Представим $s$ в виде
$s=r(1)N(1)+\ldots+r(k-1)N(k-1)+aN(k)+b$, где
$0\leqslant a\leqslant r(k)-1$,
$0\leqslant b\leqslant N(k)-1$.
Оценим
\begin{equation*}
\delta_{s}=\|x-\sum_{l=1}^{k-1}\sum_{i=1}^{r(l)}\sum_{j=1}^{N(l)}y_{l,i,j}-\sum_{i=1}^{a}\sum_{j=1}^{N(k)}y_{k,i,j}
-\sum_{j=1}^{b}y_{k,a+1,j}\|.
\end{equation*}
Из построения $y_{l,i,j}$ следуют оценки
\begin{equation}\label{E:sufconv1}
\|x-\sum_{l=1}^{k-1}\sum_{i=1}^{r(l)}\sum_{j=1}^{N(l)}y_{l,i,j}\|<2^{-(k-1)},
\end{equation}
\begin{equation}\label{E:sufconv2}
\|\sum_{i=1}^{a}\sum_{j=1}^{N(k)}y_{k,i,j}\|<2^{-(k-1)}+2^{-k},
\end{equation}
\begin{equation}\label{E:sufconv3}
\|\sum_{j=1}^{b}y_{k,a+1,j}\|< N(k)2^{-k}(N(k))^{-1}=2^{-k}.
\end{equation}
Из неравенств \eqref{E:sufconv1},\eqref{E:sufconv2},\eqref{E:sufconv3}
следует
$\delta_{s}<6\cdot 2^{-k}\to 0,\,s\to\infty$.
Поэтому справедливо равенство \eqref{E:sufconv}.
Дополняя его в нужных местах нулями, получим искомое разложение
$x=\sum_{k=1}^{\infty}x_{k}$,
$x_{k}\in X_{k}$.
\end{proof}

Приведём пример системы
$S$, для которой выполнено условие теоремы
\ref{T:SUFforRS}.
Для двух подпространств $Y,Z$ пространства $X$ определим
\begin{equation}\label{E:rho0}
\rho_{0}(Y,Z)=\sup\{d(y,Z)\mid y\in Y,\|y\|=1\}.
\end{equation}

\begin{example}
Пусть система подпространств
$Y_{j},j\in\Lambda$ ($\Lambda$ "--- некоторое множество индексов) такова, что
замыкание линейной оболочки $Y_{j},j\in\Lambda$ равно
$X$. Пусть система
$S=(X;X_{k},k\geqslant 1)$ такова, что для каждого
$j\in\Lambda$ существует последовательность натуральных чисел
$k(1)<k(2)<\ldots$, для которой
$\lim_{n\to\infty}\rho_{0}(Y_{j},X_{k(n)})=0$.
Тогда для каждого
$j\in\Lambda$
$Y_{j}\subset\Delta(S)$. Поэтому
$S$ удовлетворяет условию теоремы \ref{T:SUFforRS}, а значит, является ПСП в
$X$.
\end{example}

\begin{remark}
Если система
$S$ удовлетворяет условию теоремы
\ref{T:SUFforRS}, то для каждого
$n\in\mathbb{N}$ система
$S_{(\geqslant n)}=(X;X_{k+n-1},k\geqslant 1)$ также удовлетворяет условию теоремы
\ref{T:SUFforRS}, а поэтому является ПСП в
$X$. Поэтому для каждого
$n\in\mathbb{N}$ замыкание линейной оболочки
подпространств
$X_{k},k\geqslant n$ равно $X$.
Последнее условие не является достаточным для того, чтобы
$S$ была ПСП в $X$.
Это показывает следующий пример (который относится к математическому фольклору).

Пусть
$X=L_{p}([0,1],dx)$ ($p\in[1,\infty)$), подпространство
$X_{k}$ порождено
$x^{k},k\geqslant 0$
(нам удобнее нумеровать подпространства числами $0,1,2,\ldots$, а не
$1,2,\ldots$).
Тогда для всех
$n\geqslant 0$ замыкание линейной оболочки
$X_{k},k\geqslant n$ равно $X$, но
$S$ не есть ПСП в $X$.
Действительно, если
$f(x)=\sum_{k=0}^{\infty}a_{k}x^{k}$ (сходимость по норме пространства $X$), то
$\|a_{k}x^{k}\|\to 0,\,k\to\infty$, а
поэтому для всех достаточно больших
$k~$
$|a_{k}|\leqslant(kp+1)^{1/p}$. Поэтому
$f(x)\in C^{\infty}([0,1))$.
\end{remark}

\subsection{Об одном достаточном условии для того, чтобы $x\in X$ допускал разложение по системе подпространств $S$ в
случае гильбертова пространства $X$}

Пусть
$X$ "--- гильбертово пространство. Для подпространства
$Y\subset X$ обозначим $P_{Y}$ ортопроектор на $Y$. Пусть
$S=(X;X_{k},k\geqslant 1)$ "--- система подпространств $X$. Рассмотрим произвольный
$x\in X$ и попробуем разложить его по системе подпространств
$S$, т.е. представить в виде
$x=\sum_{k=1}^{\infty}x_{k}$, где
$x_{k}\in X_{k},k\geqslant 1$.

Естественно попробовать определить искомое разложение следующим образом:
\begin{align*}
&x=P_{X_{1}}x+(I-P_{X_{1}})x=P_{X_{1}}x+P_{X_{2}}(I-P_{X_{1}})x+(I-P_{X_{2}})(I-P_{X_{1}})x=\ldots=\\
&=\sum_{k=1}^{n}P_{X_{k}}(I-P_{X_{k-1}})\ldots(I-P_{X_{1}})x+(I-P_{X_{n}})\ldots(I-P_{X_{1}})x=\ldots
\end{align*}
Определим операторы
$E_{0}=I$,
$E_{n}=(I-P_{X_{n}})\ldots(I-P_{X_{1}}),n\geqslant 1$.
Обозначим
$x_{n}=P_{X_{n}}E_{n-1}x,n\geqslant 1$, тогда
$x=\sum_{k=1}^{n}x_{k}+E_{n}x,n\geqslant 1$.
Таким образом, получаем следующее утверждение.

\begin{proposition}
Если $E_{n}x\to 0,\,n\to\infty$, то
$x=\sum_{k=1}^{\infty}x_{k}$.
\end{proposition}

Таким образом, если
последовательность операторов
$E_{n}$ сходится к $0$ сильно при $n\to\infty$, то система подпространств
$S$ является ПСП в
$X$. Однако вопрос о сильной сходимости
$E_{n}$ к $0$ может оказаться очень сложным.
Рассмотрим следующий пример.
Пусть $N\in\mathbb{N}$,
$H_{k},1\leqslant k\leqslant N$ "--- подпространства
$X$, причём
$\bigcap_{k=1}^{N}H_{k}=0$.
Пусть отображение
$i(\cdot):\mathbb{N}\to\{1,2,\ldots,N\}$ таково, что
$i(k+1)\neq i(k),k\geqslant 1$ и для каждого
$m\in\{1,2,\ldots,N\}$ существует бесконечно много
$k$, для которых
$i(k)=m$.
Определим
$X_{k}=H_{i(k)}^{\bot},k\geqslant 1$.
Тогда
$E_{n}=P_{H_{i(n)}}\ldots P_{H_{i(1)}},n\geqslant 1$.
Известно, что $E_{n}$ сходится к $0$ слабо при $n\to\infty$ (см.\cite{AmAn}).
Вопрос о сильной сходимости
$E_{n}$ к $0$ при
$n\to\infty$ называется проблемой Гальперина и является
чрезвычайно сложным (см., например, \cite{Bauschke}).
В то же время
из теоремы \ref{T:SUFforRS} следует, что
$S$ является ПСП в
$X$.

\subsection{Устойчивость ПСП в банаховых пространствах}

Пусть
$X$ "--- банахово пространство,
$S=(X;X_{k},k\geqslant 1)$ "--- система подпространств $X$.
Мы покажем, что если подпространства
$X_{k},\widetilde{X}_{k}$ достаточно <<близки>>, $k\geqslant 1$, то система подпространств
$\widetilde{S}=(X;\widetilde{X}_{k},k\geqslant 1)$ также является ПСП в
$X$. За меру <<близости>> подпространств выберем величину
$\rho_{0}(X_{k},\widetilde{X}_{k})$, определённую формулой
\eqref{E:rho0}. Для доказательства соответствующих результатов мы
обобщим результаты параграфов 2,3 работы \cite{Slepchenko}, в которой рассматриваются системы одномерных подпространств,
на произвольные системы подпространств.

Введём необходимые определения (обобщающие определения параграфа 2 работы \cite{Slepchenko}). Для набора
$P=(x_{1},\ldots,x_{n})$, где
$x_{k}\in X_{k},1\leqslant k\leqslant n,$ определим
$\Sigma(P)=\sum_{k=1}^{n}x_{k}$, а также
\begin{equation*}
\Theta_{S}(P)=\max_{1\leqslant k\leqslant n}\|\sum_{j=1}^{k}x_{j}\|.
\end{equation*}
Для
$x\in X,\varepsilon>0$ определим
\begin{equation*}
\Theta_{S}(x,\varepsilon)=\inf\{\Theta_{S}(P)\mid \|\Sigma(P)-x\|\leqslant\varepsilon\}.
\end{equation*}
(Мы считаем, что $\inf(\varnothing)=\infty$.)
Для
$x\in X$ определим
\begin{equation*}
\Theta_{S}^{*}(x)=\sup\{\Theta_{S}(x,\varepsilon)\mid\varepsilon>0\}=\lim_{\varepsilon\to 0+}\Theta_{S}(x,\varepsilon).
\end{equation*}
Наконец, определим
\begin{equation*}
\overline{\Theta}_{S}=\sup\{\Theta_{S}^{*}(x)\mid \|x\|\leqslant 1\}.
\end{equation*}

Следующие две леммы и теорема доказываются точно так же, как леммы 1,2 и теорема 1 в
\cite{Slepchenko}.

\begin{lemma}
Если
$\Theta_{S}^{*}(x)<\infty$ для произвольного
$x\in X$, то
$\Theta_{S}^{*}(x)$ "--- норма на
$X$, эквивалентная
$\|\cdot\|$.
\end{lemma}

\begin{lemma}
Если для некоторых
$\alpha\in(0,1),B>0$ выполнено
$\Theta_{S}(x,\alpha\|x\|)\leqslant B\|x\|,x\in X$, то
$\Theta_{S}^{*}(x)\leqslant\frac{B}{1-\alpha}\|x\|,x\in X$.
\end{lemma}

\begin{theorem}
Следующие утверждения эквивалентны:
\begin{enumerate}
\item
$S$ является ПСП в $X$,
\item
существуют
$\alpha\in(0,1),B>0$, такие, что для произвольного
$x\in X$\\
$\Theta_{S}(x,\alpha\|x\|)\leqslant B\|x\|$,
\item
$\Theta_{S}^{*}(x)<\infty$ для произвольного $x\in X$,
\item
$\overline{\Theta}_{S}<\infty$.
\end{enumerate}
\end{theorem}

Теперь установим теорему об устойчивости ПСП в $X$.

\begin{theorem}
Пусть
$S$ "--- ПСП в $X$. Если система подпространств
$\widetilde{S}=(X;\widetilde{X}_{k},k\geqslant 1)$ такова, что
$\sum_{k=1}^{\infty}\rho_{0}(X_{k},\widetilde{X}_{k})<(2\overline{\Theta}_{S})^{-1}$, то
$\widetilde{S}$ является ПСП в
$X$.
\end{theorem}
\begin{proof}
Доказательство аналогично доказательству теоремы 2 в
\cite{Slepchenko}.
Пусть
$x\in X,x\neq 0$. Пусть
$\varepsilon>0$. Существует
$P=(x_{1},\ldots,x_{n}),x_{k}\in X_{k}$, такой, что
\begin{equation*}
\|x-\Sigma(P)\|\leqslant\varepsilon,\,
\Theta_{S}(P)\leqslant\overline{\Theta}_{S}(1+\varepsilon)\|x\|.
\end{equation*}
Тогда
$\|x_{k}\|\leqslant 2\overline{\Theta}_{S}(1+\varepsilon)\|x\|,1\leqslant k\leqslant n.$
Обозначим $d_{k}=\rho_{0}(X_{k},\widetilde{X}_{k}),k\geqslant 1$.
Для произвольного
$k,1\leqslant k\leqslant n,$ существует
$\widetilde{x}_{k}\in\widetilde{X}_{k}$, для которого
$\|x_{k}-\widetilde{x}_{k}\|\leqslant d_{k}(1+\varepsilon)\|x_{k}\|$.
Определим
$\widetilde{P}=(\widetilde{x}_{1},\ldots,\widetilde{x}_{n})$. Тогда
\begin{equation*}
\|x-\Sigma(\widetilde{P})\|\leqslant\|x-\Sigma(P)\|+\|\Sigma(P)-\Sigma(\widetilde{P})\|\leqslant
\varepsilon+\sum_{k=1}^{n}d_{k}(1+\varepsilon)\|x_{k}\|\leqslant \varepsilon+2\overline{\Theta}_{S}(1+\varepsilon)^{2}\|x\|\sum_{k=1}^{n}d_{k}.
\end{equation*}
Для произвольного
$m,1\leqslant m\leqslant n,$ имеем
\begin{equation*}
\|\sum_{k=1}^{m}\widetilde{x}_{k}\|\leqslant\|\sum_{k=1}^{m}x_{k}\|+\|\sum_{k=1}^{m}(\widetilde{x}_{k}-x_{k})\|\leqslant
\overline{\Theta}_{S}(1+\varepsilon)\|x\|+2\overline{\Theta}_{S}(1+\varepsilon)^{2}\|x\|\sum_{k=1}^{m}d_{k}.
\end{equation*}

Зафиксируем
$\alpha\in(2\overline{\Theta}_{S}\sum_{k=1}^{\infty}d_{k},1)$,
$B>\overline{\Theta}_{S}+2\overline{\Theta}_{S}\sum_{k=1}^{\infty}d_{k}$.
При достаточно малом
$\varepsilon>0$ из доказаных неравенств имеем
$\Theta_{\widetilde{S}}(x,\alpha\|x\|)\leqslant B\|x\|$.
Поэтому
$\widetilde{S}$ является ПСП в $X$.
\end{proof}

\section{Абсолютно представляющие системы подпространств в банаховых пространствах}\label{S:ARSofSub}

\subsection{Критерии АПСП}

Пусть
$X$ "--- банахово пространство,
$S=(X;X_{k},k\geqslant 1)$ "--- система подпространств $X$.
Определим
$l_{1}(X_{1},X_{2},\ldots)$ как множество последовательностей
$\xi=(x_{1},x_{2},\ldots),x_{k}\in X_{k}$, для которых
$\|\xi\|=\sum_{k=1}^{\infty}\|x_{k}\|<\infty$. Ясно, что
$l_{1}(X_{1},X_{2},\ldots)$ "--- банахово пространство. Определим оператор
$A:l_{1}(X_{1},X_{2},\ldots)\to X$ равенством
$A(x_{1},x_{2},\ldots)=\sum_{k=1}^{\infty}x_{k}$.
Система
$S$ является АПСП в $X$ тогда и только тогда, когда
$\mathrm{Im}(A)=X$. Хорошо известно, что это равносильно тому, что
$A^{*}:X^{*}\to (l_{1}(X_{1},X_{2},\ldots))^{*}$ является изоморфным вложением, т.е.
для некоторого
$\varepsilon>0$
$\|A^{*}\varphi\|\geqslant\varepsilon\|\varphi\|,\varphi\in X^{*}$.
Легко видеть, что
$(l_{1}(X_{1},X_{2},\ldots))^{*}=l_{\infty}(X_{1}^{*},X_{2}^{*},\ldots)$ "--- множество всех последовательностей
$\eta=(\varphi_{1},\varphi_{2},\ldots),\varphi_{k}\in X_{k}^{*}$, для которых
$\|\eta\|=\sup_{k}\|\varphi_{k}\|<\infty$. При этом действие
$\eta(\xi)=\sum_{k=1}^{\infty}\varphi_{k}(x_{k})$.
Легко видеть, что
$A^{*}\varphi=(\varphi^{(S,1)},\varphi^{(S,2)},\ldots)$, где
$\varphi^{(S,k)}$ обозначено сужение $\varphi$ на
$X_{k}$. Таким образом, получаем следующую теорему.

\begin{theorem}\label{T:ARSofSubs}
$S$ является АПСП в $X$ тогда и только тогда, когда существует
$\varepsilon>0$ такое, что
\begin{equation}\label{E:ARSofSubs}
\sup_{k}\|\varphi^{(S,k)}\|\geqslant\varepsilon\|\varphi\|,\varphi\in X^{*}.
\end{equation}
\end{theorem}

Отметим, что теорему \ref{T:ARSofSubs} можно сформулировать следующим образом
(уменьшив $\varepsilon$):
$S$ является АПСП в $X$ тогда и только тогда, когда существует
$\varepsilon>0$ такое, что для произвольного
$\varphi\in X^{*},\|\varphi\|=1$ существуют
$k\geqslant 1,x\in X_{k},\|x\|=1$ такие, что
$|\varphi(x)|\geqslant\varepsilon$.

Понятие АПСП тесно связано с понятием абсолютно представляющего семейства (АПСм).
Напомним (см., например,
\cite{KorARSs}), что множество
$D\subset X$ называется АПСм в
$X$ если для произвольного
$x\in X$ существуют
$a_{j}\in\mathbb{K}$,
$x_{j}\in D$, такие, что
$x=\sum_{j=1}^{\infty}a_{j}x_{j}$ и
$\sum_{j=1}^{\infty}\|a_{j}x_{j}\|<\infty$.
АПСм в банаховых и гильбертовых пространствах изучались в
\cite{Vereng},\cite{Verrus},\cite{Shraifel}. Для системы подпространств
$S$ определим
$D(S)=\bigcup_{k=1}^{n}\{x\in X_{k},\|x\|=1\}$. Ясно, что
$S$ является АПСП в $X$ тогда и только тогда, когда
$D(S)$ является АПСм в $X$.

Приведенный критерий для АПСП (см. абзац после теоремы \ref{T:ARSofSubs}) можно получить из
следующего хорошо известного критерия для АПСм
(см., например, теорему 1 в \cite{KorARSs}, теорему 2.1 в
\cite{Vereng}), который доказывается аналогично теореме \ref{T:ARSofSubs}.
(Множество $D$ называется нормированным, если $\|x\|=1,x\in D$.)

\begin{theorem}\label{T:CritARS}
Пусть $D$ "--- нормированное множество в $X$.
$D$ является АПСм в $X$ тогда и только тогда, когда существует
$\varepsilon>0$ такое, что для произвольного
$\varphi\in X^{*},\|\varphi\|=1,$ существует
$x\in D$ такой, что
$|\varphi(x)|\geqslant\varepsilon$.
\end{theorem}

Далее мы докажем критерий для АПСм в
$X$, из которого непосредственно следует критерий для
АПСП в
$X$ (вместо
$D$ надо взять $D(S)$).
Следующая теорема обобщает теорему 3 в
\cite{Shraifel} и показывает, что в теореме \ref{T:CritARS} условие произвольности
$\varphi\in X^{*}$ можно ослабить.
(Будем говорить, что $D$ тотально в $X$, если замыкание линейной оболочки $D$ равно $X$).

\begin{theorem}\label{T:genShr}
Пусть $X$ "--- банахово пространство,
$Y$ "--- конечномерное подпространство $X$,
$D$ "--- тотальное в $X$ нормированное множество.
$D$ является АПСм в $X$ тогда и только тогда, когда существует
$\varepsilon>0$ такое, что для произвольного
$\varphi\in Y^{\bot},\|\varphi\|=1$ существует
$x\in D$ такой, что
$|\varphi(x)|\geqslant\varepsilon$.
\end{theorem}

Для доказательства нам нужна следующая лемма (которая относится к математическому фольклору).

\begin{lemma}
Пусть
$Y,Z$ "--- подпространства банахова пространства $X$. Если
$Y\cap Z=0$ и
$Y+Z=X$, то существует
$c>0$ такое, что для произвольных
$y\in Y,z\in Z$
$\|y+z\|\geqslant c(\|y\|+\|z\|)$.
\end{lemma}
\begin{proof}
Определим пространство
$Y\oplus Z$ как множество пар
$\xi=(y,z),y\in Y,z\in Z,$ с нормой
$\|\xi\|=\|y\|+\|z\|$. Ясно, что
$Y\oplus Z$ банахово. Определим оператор
$A:Y\oplus Z\to X$ равенством
$A(y,z)=y+z$. Тогда
$A$ ограничен,
$\ker A=0$,
$\mathrm{Im} A=X$. Поэтому
$A$ обратим, откуда непосредственно следует нужное утверждение.
\end{proof}

\begin{proof}[Доказательство теоремы \ref{T:genShr}]
Необходимость очевидна. Докажем достаточность. Предположим, что
$D$ "--- не АПСм в $X$.
Поскольку
$Y$ конечномерно, то оно дополняемо в
$X$, т.е. существует подпространство
$Z$ такое, что
$Y\cap Z=0$ и
$Y+Z=X$. Существует
$c_{1}>0$ такое, что
\begin{equation}\label{E:sumclosed}
\|y+z\|\geqslant c_{1}(\|y\|+\|z\|),\,y\in Y,z\in Z.
\end{equation}
Выберем в
$Y$ нормированный базис
$e_{1},\ldots,e_{m}$. Существует
$c_{2}>0$ такое, что для произвольных
$t_{1},\ldots,t_{m}\in\mathbb{K}$
$\|\sum_{k=1}^{m}t_{k}e_{k}\|\geqslant c_{2}\sum_{k=1}^{m}|t_{k}|$.

Рассмотрим произвольное
$\delta>0$. Поскольку $D$ тотально в $X$, существуют элементы
$f_{1},\ldots,f_{m}$ из линейной оболочки $D$, такие, что
$\|e_{k}-f_{k}\|<\delta,\|f_{k}\|=1$ для
$k=1,\ldots,m$. Ясно, что
$\{f_{1},\ldots,f_{m}\}\cup D$ не является АПСм в
$X$. Поэтому существует
$\varphi\in X^{*}$,
$\|\varphi\|=1$ такой, что
$|\varphi(f_{k})|\leqslant\delta$ для
$1\leqslant k\leqslant m$,
$|\varphi(x)|\leqslant\delta$ для
$x\in D$. Определим линейные функционалы
$\psi,\eta$ равенствами
$\psi(y+z)=\varphi(z),$
$\eta(y+z)=\varphi(y),$
$y\in Y,z\in Z$. Из неравенства
\eqref{E:sumclosed} следует, что
$\psi,\eta\in X^{*}$. Более того,
$\psi\in Y^{\bot}$,
$\eta\in Z^{\bot}$. 

Оценим
$\|\eta\|$. Пусть
$y\in Y$,
$y=\sum_{k=1}^{m}t_{k}e_{k}$. Тогда
\begin{equation*}
|\varphi(y)|\leqslant\sum_{k=1}^{m}|t_{k}||\varphi(e_{k})|\leqslant\sum_{k=1}^{m}|t_{k}|(|\varphi(e_{k}-f_{k})|+|\varphi(f_{k})|)
\leqslant 2\delta\sum_{k=1}^{m}|t_{k}|\leqslant 2\delta c_{2}^{-1}\|y\|.
\end{equation*}
Поэтому для произвольных $y\in Y,z\in Z$
\begin{equation*}
|\eta(y+z)|=|\varphi(y)|\leqslant 2\delta c_{2}^{-1}\|y\|\leqslant 2c_{1}^{-1}c_{2}^{-1}\delta\|y+z\|.
\end{equation*}
Положим
$c_{3}=2c_{1}^{-1}c_{2}^{-1}$, тогда
$\|\eta\|\leqslant c_{3}\delta$.

Поскольку
$\|\varphi\|=1$, то
$\|\psi\|\geqslant (1-c_{3}\delta)$. Для произвольного
$x\in D$
$|\psi(x)|\leqslant\delta(1+c_{3})$. Положим
$\widetilde{\psi}=\psi/\|\psi\|$. Тогда
$\widetilde{\psi}\in Y^{\bot},$
$\|\widetilde{\psi}\|=1$. Для произвольного
$x\in D$
$|\widetilde{\psi}(x)|\leqslant \delta(1+c_{3})/(1-c_{3}\delta)$.
При достаточно малых
$\delta$ получим противоречие.
\end{proof}

Рассмотрим АПСП в равномерно гладких пространствах
(АПСм в равномерно гладких пространствах изучались в \cite{Vereng},\cite{Verrus}).
Напомним определение равномерно гладкого пространства. Определим модуль гладкости пространства
$X$ равенством
\begin{equation*}
\rho(\tau)=\sup\{(\|x+y\|+\|x-y\|)/2 -1\mid \|x\|=1,\|y\|=\tau\},\,\tau>0.
\end{equation*}
$X$ называется равномерно гладким если
$\rho(\tau)/\tau\to 0$ при $\tau\to 0$.
Для нас равномерно гладкие пространства важны по следующей причине:
как мы увидим при доказательстве следующей теоремы, если
$S$ "--- АПСП в равномерно гладком $X$, то для каждого
$x\in X$ разложение
$x$ в абсолютно сходящийся ряд по системе подпространств
$S$ может быть получено простым <<естественным>> образом.

Будем говорить, что множество
$A$ является $\lambda$-сетью для множества $B$, если для произвольного
$x\in B$ существует $y\in A$ такой, что $\|x-y\|\leqslant\lambda$. Обозначим
$V_{X}=\{x\in X,\|x\|=1\}.$ Напомним, что для системы подпространств $S=(X;X_{k},k\geqslant 1)$ $D(S)=\bigcup_{k=1}^{\infty}V_{X_{k}}$.

\begin{theorem}\label{T:unifsmooth}
Пусть $X$ "--- равномерно гладкое банахово пространство. Тогда утверждения равносильны:
\begin{enumerate}
\item
$S=(X;X_{k},k\geqslant 1)$ является АПСП в $X$,
\item
существуют
$\tau,\lambda\in(0,1)$ такие, что
$\tau D(S)$ "--- $\lambda$-сеть для $V_{X}$,
\item
$\lambda_{S}=\sup_{\|x\|=1}\inf_{k\geqslant 1}d(x,X_{k})<1.$
\end{enumerate}
\end{theorem}
\begin{proof}
$(1)\Rightarrow (2).$
Для действительного пространства
$X$ нужное утверждение следует из теоремы 3.1 в
\cite{Vereng}. Для комплексного
$X$ рассмотрим
$X$ как пространство над $\mathbb{R}$ и из упомянутой теоремы получим нужное.

$(2)\Rightarrow (3).$ Очевидно.

$(3)\Rightarrow (1).$ Доказательство аналогично доказательству теоремы
3 в \cite{Verrus}. Фиксируем
$\lambda\in(\lambda_{S},1)$. Рассмотрим произвольный
$x\in X$. Существуют
$i(1)\in\mathbb{N}$,
$x_{1}\in X_{i(1)}$ такие, что
$\|x-x_{1}\|\leqslant\lambda\|x\|$. Определим
$y_{1}=x-x_{1}$, тогда
$\|y_{1}\|\leqslant\lambda\|x\|$ и
$x=x_{1}+y_{1}$. Далее проделаем аналогичную процедуру. Пусть мы имеем разложение
$x=x_{1}+\ldots+x_{k}+y_{k}$. Существуют
$i(k+1)\in\mathbb{N}$,
$x_{k+1}\in X_{i(k+1)}$ такие, что
$\|y_{k}-x_{k+1}\|\leqslant\lambda\|y_{k}\|$. Определим
$y_{k+1}=y_{k}-x_{k+1}$, тогда
$\|y_{k+1}\|\leqslant\lambda\|y_{k}\|$ и
$x=x_{1}+\ldots+x_{k+1}+y_{k+1}$. Индукцией по
$k$ легко установить, что
$\|y_{k}\|\leqslant\lambda^{k}\|x\|,k\geqslant 1$. Поэтому
\begin{equation}\label{E:absconvser}
x=\sum_{k=1}^{\infty}x_{k}.
\end{equation}
Ясно, что
$\|x_{k}\|\leqslant\lambda^{k-1}(1+\lambda)\|x\|,k\geqslant 1$, поэтому
$\sum_{k=1}^{\infty}\|x_{k}\|\leqslant((1+\lambda)/(1-\lambda))\|x\|.$
Для
$k\geqslant 1$ определим
$z_{k}=\sum_{j:i(j)=k}x_{j}$, тогда
$z_{k}\in X_{k}$. Из
\eqref{E:absconvser} имеем 
$x=\sum_{k=1}^{\infty}z_{k}$.
\end{proof}

\begin{remark}
В работе
\cite{Verrus} доказано, что каждая АПС (одномерных подпространств) в равномерно гладком
$X$ является
<<быстрой>> ПС
(см. определение 3 и теорему 3 в \cite{Verrus}).
Аналогичное утверждение верно для АПСП. Как следует из доказательства
теоремы
\ref{T:unifsmooth},
$(3)\Rightarrow(1)$, каждая АПСП $S$ в равномерно гладком
$X$ является 
<<быстрой>> ПСП (наше определение согласовано с определением 3 в \cite{Verrus}):
существуют
$C>0$ и $\lambda\in(0,1)$ такие, что для произвольного
$x\in X$ существует инъективное отображение
$k\mapsto n(k)$ и элементы
$y_{k}\in X_{n(k)}$ такие, что
$x=\sum_{k=1}^{\infty}y_{k}$ и
$\|y_{k}\|\leqslant C\lambda^{k}\|x\|,k\geqslant 1$.
\end{remark}

Рассмотрим теперь АПСП в гильбертовых пространствах; как мы увидим далее, критерий для АПСП в гильбертовых 
пространствах приобретает геометрическую наглядность (см. также теоремы 1,2 в 
\cite{Shraifel}). Итак, пусть
$X$ гильбертово. Тогда
$X^{*}$ можно отождествить с 
$X$:
$\varphi(\cdot)=(\cdot,\varphi)$.
$S$ является АПСП в $X$ тогда и только тогда, когда существует
$\varepsilon>0$ такое, что для произвольного
$\varphi\in V_{X}$ существует
$x\in D(S)$ такой, что
$|(x,\varphi)|\geqslant\varepsilon.$

\begin{theorem}
Пусть
$\tau>0,\varepsilon\in(0,1]$. Следующие условия равносильны:
\begin{enumerate}
\item
для произвольного $\varphi\in V_{X}$ существует $x\in D(S)$ такой, что $|(x,\varphi)|\geqslant\varepsilon$,
\item
$\tau D(S)$ является $\sqrt{1+\tau^{2}-2\tau\varepsilon}$-сетью для $V_{X}$.
\end{enumerate}
\end{theorem}
\begin{proof}
Для произвольных
$\varphi\in V_{X},x\in D(S)$ имеем
$\|\varphi-\tau x\|^{2}=1+\tau^{2}-2\tau\mathrm{Re}(x,\varphi).$
Из этого равенства очевидным образом следует нужное утверждение.
\end{proof} 

\begin{corollary}
Пусть
$\tau>0$. Система подпространств
$S$ является АПСП в
$X$ тогда и только тогда, когда существует
$\lambda\in(0,\sqrt{1+\tau^{2}})$ такое, что
$\tau D(S)$ является
$\lambda$-сетью для $V_{X}$.
\end{corollary}

\subsection{Об одном свойстве АПСП}

Перед тем, как сформулировать и доказать следующую теорему, напомним определение
$C$-выпуклого пространства и некоторые свойства
$C$-выпуклых пространств.

Пусть
$Y$ "--- банахово пространство над $\mathbb{R}$. Обозначим
$c_{0}(\mathbb{R})$ множество последовательностей
$\xi=(z_{1},z_{2},\ldots),z_{k}\in\mathbb{R}$, для которых
$z_{k}\to 0$ при $k\to\infty$, с нормой
$\|\xi\|=\sup_{k}|z_{k}|$.
$Y$ называется $C$-выпуклым, если
$c_{0}(\mathbb{R})$ не является финитно представимым в $Y$ (см., например, параграфы 5.1, 5.2 книги \cite{Kadets}
и библиографию в конце параграфа 5.2). Для натурального
$n$ определим
\begin{equation}\label{E:Cconv}
C(n,Y)=\inf\left\{\max\left\{\|\sum_{k=1}^{n}\alpha_{k}y_{k}\| \mid \alpha_{k}=\pm 1,1\leqslant k\leqslant n\right\} \mid
y_{k}\in Y,\|y_{k}\|\geqslant 1,1\leqslant k\leqslant n\right\}.
\end{equation}
В из лемм 5.2.1, 5.2.2 и теоремы 5.2.2 книги
\cite{Kadets} следует, что
$Y$
$C$-выпукло тогда и только тогда, когда
$C(n,Y)\to\infty$ при $n\to\infty$.

Пусть
$Y$ "--- банахово пространство над $\mathbb{C}$. Обозначим
$c_{0}(\mathbb{C})$ множество последовательностей
$\xi=(z_{1},z_{2},\ldots),z_{k}\in\mathbb{C}$, для которых
$z_{k}\to 0$ при $k\to\infty$, с нормой
$\|\xi\|=\sup_{k}|z_{k}|$.
$Y$ называется $C$-выпуклым, если
$c_{0}(\mathbb{C})$ не является финитно представимым в $Y$. Для натурального
$n$ определим
$C(n,Y)$ формулой \eqref{E:Cconv}; величину
$C_{\mathbb{C}}(n,Y)$ формулой
\eqref{E:Cconv}, только максимум берётся по
$|\alpha_{k}|=1,\alpha_{k}\in\mathbb{C}$.
Перенося леммы 5.2.1, 5.2.2 и теорему 5.2.2 книги
\cite{Kadets} на случай комплексных пространств, получим, что
$Y$
$C$-выпукло тогда и только тогда, когда $C_{\mathbb{C}}(n,Y)\to\infty$ при $n\to\infty$.
Поскольку для произвольных
$\alpha_{1},\ldots,\alpha_{n}\in\mathbb{C},|\alpha_{k}|\leqslant 1$ и произвольных
$y_{1},\ldots,y_{n}\in Y$
\begin{equation*}
\|\sum_{k=1}^{n}\alpha_{k}y_{k}\|\leqslant\|\sum_{k=1}^{n}\mathrm{Re}(\alpha_{k})y_{k}\|+\|\sum_{k=1}^{n}\mathrm{Im}(\alpha_{k})y_{k}\|
\leqslant 2\max_{\beta_{k}=\pm 1}\|\sum_{k=1}^{n}\beta_{k}y_{k}\|,
\end{equation*}
то $C_{\mathbb{C}}(n,Y)\leqslant 2C(n,Y)$, а поэтому
$Y$
$C$-выпукло тогда и только тогда, когда
$C(n,Y)\to\infty$ при
$n\to\infty$.

Для
$I\subset\mathbb{N}$ обозначим
$\min(I)$ наименьший элемент множества
$I$.

\begin{theorem}\label{T:ARSSCconv}
Предположим, что $X^{*}$ $C$-выпукло.
Если
$S$ "--- АПСП в $X$, то существует
$N_{0}$, такое, что для произвольного конечного
$I\subset\mathbb{N}$, удовлетворяющего
$\min(I)\geqslant N_{0}$, система подпространств
$X_{k},k\notin I$ является АПСП в
$X$.
\end{theorem}

Теорема
\ref{T:ARSSCconv} очевидным образом следует из следующей леммы.

\begin{lemma}
Пусть
$I_{1},I_{2},\ldots$ "--- подмножества $\mathbb{N}$,
$m\in\mathbb{N}$. Предположим, каждое натуральное
$n$ принадлежит не более чем
$m$ множествам $I_{j}$. Тогда для некоторого
$j$ система подпространств $X_{k},k\notin I_{j}$ является АПСП в
$X$.
\end{lemma}
\begin{proof}
Существует $\varepsilon>0$, для которого выполнено \eqref{E:ARSofSubs}.
Предположим, утверждение леммы неверно. Тогда для произвольного
$j$ существует
$\varphi_{j}\in X^{*},\|\varphi_{j}\|=1$, такой, что
$\|\varphi_{j}^{(S,k)}\|\leqslant 2^{-j}$ для всех
$k\notin I_{j}$. Рассмотрим произвольное
$n\in\mathbb{N}.$ Существуют
$\alpha_{1},\ldots,\alpha_{n}\in\{\pm 1\}$, для которых
$\|\sum_{j=1}^{n}\alpha_{j}\varphi_{j}\|\geqslant C(n,X^{*}).$
Определим
$\varphi=\sum_{j=1}^{n}\alpha_{j}\varphi_{j}$. Тогда
$\|\varphi\|\geqslant C(n,X^{*})$ и для произвольного
$k$
$\|\varphi^{(S,k)}\|\leqslant\sum_{j=1}^{n}\|\varphi_{j}^{(S,k)}\|\leqslant m+1$.
Из неравенства \eqref{E:ARSofSubs} следует
$m+1\geqslant\varepsilon C(n,X^{*})$, что, в силу произвольности
$n$, противоречит
$C$-выпуклости
$X^{*}$.
\end{proof}

\subsection{Устойчивость АПСП}

\begin{theorem}
Если $S$ является АПСП в $X$, то существует
$\delta>0$ такое, что произвольная система
$\widetilde{S}=(X;\widetilde{X}_{k},k\geqslant 1)$, удовлетворяющая
$\sup_{k}\rho_{0}(X_{k},\widetilde{X}_{k})<\delta,$ является
АПСП в $X$.
\end{theorem}
\begin{proof}
Пусть
$\varepsilon>0$ такое, что выполнено неравенство
\eqref{E:ARSofSubs}. Покажем, что если для системы подпространств
$\widetilde{S}$
$d=\sup_{k}\rho_{0}(X_{k},\widetilde{X}_{k})<\varepsilon$, то
$\widetilde{S}$ является АПСП в
$X$. Выберем
$\varepsilon_{1}<\varepsilon,d_{1}>d$, причём
$d_{1}<\varepsilon_{1}$. Рассмотрим произвольный
$\varphi\in X^{*},\|\varphi\|=1.$ Существует
$k\in\mathbb{N}$ и
$x\in X_{k},\|x\|=1$, такие, что
$|\varphi(x)|\geqslant\varepsilon_{1}$. Существует
$\widetilde{x}\in\widetilde{X}_{k}$, для которого
$\|x-\widetilde{x}\|\leqslant d_{1}$. Имеем
\begin{equation*}
|\varphi(\widetilde{x})|\geqslant|\varphi(x)|-|\varphi(x-\widetilde{x})|\geqslant(\varepsilon_{1}-d_{1}),\,
\|\widetilde{x}\|\leqslant(1+d_{1}),
\end{equation*}
откуда
$\|\varphi^{(\widetilde{S},k)}\|\geqslant(\varepsilon_{1}-d_{1})/(1+d_{1})$.
Поэтому
$\widetilde{S}$ "--- АПСП в
$X$.
\end{proof}

\end{document}